\documentclass[a4paper]{amsart}
\usepackage[all]{xy}
\SelectTips{eu}{}
\usepackage{ifthen}
\usepackage{amssymb}
\usepackage{amsthm}
\usepackage{hyperref}
\calclayout 

\title{
Krull-Schmidt categories and projective covers}
\author[Henning Krause]{Henning Krause}
\address{Henning Krause\\ Fakult\"at f\"ur Mathematik\\
Universit\"at Bielefeld\\ 33501 Bielefeld\\ Germany.}
\email{hkrause@math.uni-bielefeld.de}

\thanks{Version from October 10, 2014.}

\theoremstyle{plain}
\newtheorem{lem}{Lemma}[section]
\newtheorem{prop}[lem]{Proposition}
\newtheorem{cor}[lem]{Corollary}
\newtheorem{thm}[lem]{Theorem}

\theoremstyle{remark}
\newtheorem{rem}[lem]{Remark}

\theoremstyle{definition}
\newtheorem{exm}[lem]{Example}

\numberwithin{equation}{section}

\newcommand{\smatrix}[1]{\left[\begin{smallmatrix}#1\end{smallmatrix}\right]}

\renewcommand{\mod}{\operatorname{mod}\nolimits}

\newcommand{\proj}{\operatorname{proj}\nolimits}
\newcommand{\rad}{\operatorname{rad}\nolimits}
\newcommand{\Rad}{\operatorname{Rad}\nolimits}
\newcommand{\add}{\operatorname{add}\nolimits}

\newcommand{\Mod}{\operatorname{Mod}\nolimits}

\newcommand{\End}{\operatorname{End}\nolimits}
\newcommand{\Hom}{\operatorname{Hom}\nolimits}

\renewcommand{\Im}{\operatorname{Im}\nolimits}
\newcommand{\Ker}{\operatorname{Ker}\nolimits}
\newcommand{\Coker}{\operatorname{Coker}\nolimits}

\newcommand{\id}{\operatorname{id}\nolimits}

\newcommand{\op}{\mathrm{op}}

\newcommand{\lto}[1][{}]{\stackrel{#1}{\longrightarrow}} 
\renewcommand{\to}[1][{}]{\stackrel{#1}{\rightarrow}} 
\newcommand{\xto}{\xrightarrow}

\def\a{\alpha}
\def\b{\beta}
\def\e{\varepsilon}

\def\p{\phi}
\def\r{\rho}
\def\s{\sigma}
\def\t{\tau}

\def\Ga{\Gamma}
\def\La{\Lambda}

\def\A{{\mathcal A}}

\def\C{{\mathcal C}}

\def\V{{\mathcal V}}

\def\bbZ{\mathbb Z}

\def\frI{\mathfrak I}
\def\fra{\mathfrak a}
\def\frm{\mathfrak m}
\begin{document}

\begin{abstract}
  Krull-Schmidt categories are additive categories such that each
  object decomposes into a finite direct sum of indecomposable objects
  having local endomorphism rings. We provide a self-contained
  introduction which is based on the concept of a projective cover.
\end{abstract}

\maketitle 

\setcounter{tocdepth}{1}
\tableofcontents

\section{Introduction}

Krull-Schmidt categories are ubiquitous in algebra and geometry; they
are additive categories such that each object decomposes into a finite
direct sum of indecomposable objects having local endomorphism
rings. Such  decompositions are essentially unique. Important examples
are categories of modules having finite composition length.

The aim of this note is to explain the concept of a Krull-Schmidt
category in terms of projective covers. For instance, the uniqeness of
direct sum decompositions in Krull-Schmidt categories follows from the
uniqueness of projective covers (Theorem~\ref{th:unique}).  The
exposition is basically self-contained. The results are somewhat
classical, but it seems hard to find the material in the literature.

The term `Krull-Schmidt category' refers to a result known as
`Krull-Remak-Schmidt theorem'. This formulates the existence and
uniqueness of the decomposition of a finite length module into
indecomposable ones \cite{Kr1925,Re1911,Sc1913}. Atiyah \cite{At1956}
established an analogue for coherent sheaves which is based on a chain
condition for objects of an abelian category (Theorem~\ref{th:bichain}).

The abstract concept of a Krull-Schmidt category can be found, for
example, in expositions of Auslander \cite{Au1971,Au1974} and
Gabriel-Roiter \cite{GR1997}. The basic idea is always to translate
properties of an additive category into properties of modules over
some appropriate endomorphism ring. Thus we see that an additive
category is a Krull-Schmidt category if and only if it has split
idempotents and the endomorphism ring of every object is semi-perfect
(Corollary~\ref{co:semi-perfect}). Essential ingredients of this
discussion are the radical of an additive category \cite{Ke1964} and
the concept of a projective cover \cite{Ba1960}.

\section{Additive categories and the radical}

\subsection*{Products and coproducts}
Let $\A$ be a category. A \emph{product} of a family $(X_i)_{i\in I}$
of objects of $\A$ is an object $X$ together with morphisms
$\pi_i\colon X\to X_i$ $(i\in I)$ such that for each object $A$ and
each family of morphisms $\p_i\colon A\to X_i$ $(i\in I)$ there exists
a unique morphism $\p\colon A\to X$ with $\p_i=\pi_i\p$ for all $i$.
The product solves a `universal problem' and is therefore unique up to
a unique isomorphism; it is denoted by $\prod_{i\in I}X_i$ and is
characterized by the fact that the $\pi_i$ induce a bijection
\[\Hom_\A(A,\prod_{i\in I}X_i)\xto{\sim}\prod_{i\in
I}\Hom_\A(A,X_i),\]
where the second product is taken in the category of sets.

The \emph{coproduct} $\coprod_{i\in I}X_i$ is the dual notion; it
comes with morphisms $\iota_i\colon X_i\to \coprod_{i\in I}X_i$ which
induce a bijection
\[\Hom_\A(\coprod_{i\in I}X_i,A)\xto{\sim}\prod_{i\in I}\Hom_\A(X_i,A).\]

\subsection*{Additive categories}
A category $\A$ is {\em additive} if
\begin{enumerate}
\item each morphism set $\Hom_\A(X,Y)$ is an abelian group and the
composition maps
\[\Hom_\A(Y,Z)\times\Hom_\A(X,Y)\lto\Hom_\A(X,Z)\] are bilinear,
\item there is a \emph{zero object} $0$, that is, 
$\Hom_\A(X,0)=0=\Hom_\A(0,X)$ for every object $X$, and
\item every pair of objects $X,Y$
admits a product $X\prod Y$.
\end{enumerate}

\subsection*{Direct sums}
Let $\A$ be an additive category. Given a finite number of objects
$X_1,\ldots,X_r$ of $\A$, a \emph{direct sum}
\[X=X_1\oplus\ldots\oplus X_r\] is by definition an object $X$ together with
morphisms $\iota_i\colon X_i\to X$ and $\pi_i\colon X\to X_i$ for
$1\leq i\leq r$ such that $\sum_{i=1}^r\iota_i\pi_i=\id_X$ and
$\pi_i\iota_i=\id_{X_i}$ for all $i$. 

\begin{lem}\label{le:sum}
The morphisms $\iota_i$ and $\pi_i$ induce isomorphisms
\[
\coprod_{i=1}^r X_i\cong\bigoplus_{i=1}^r X_i\cong \prod_{i=1}^r X_i. 
\]
\end{lem}
\begin{proof}
A morphism $X\to Y$ in $\A$ is an isomorphism if it induces for each
object $A$ an isomorphism $\Hom_\A(A,X)\to\Hom_\A(A,Y)$ of abelian
groups.  The functor $\Hom_\A(A,-)$ sends the direct sum
$\bigoplus_iX_i$ in $\A$ to a direct sum $\bigoplus_i\Hom_\A(A,X_i)$
of abelian groups. It is a standard fact that finite direct sums and
products of abelian groups are isomorphic.  Thus the following
composite is in fact an isomorphism.
\[\bigoplus_i\Hom_\A(A,X_i)\xto{\sim}\Hom_\A(A,\bigoplus_iX_i)\to
\Hom_\A(A,\prod_iX_i)\xto{\sim}\prod_i\Hom_\A(A,X_i).\] This
establishes the isomorphism $\bigoplus_i X_i\cong \prod_i X_i$ and the
other isomorphism $\coprod_i X_i\cong\bigoplus_i X_i$ follows by
symmetry.
\end{proof}

Lemma~\ref{le:sum} implies that a direct sum of
$X_1,\ldots,X_r$ is unique up to a unique isomorphism.  Thus one may
speak of \emph{the} direct sum and the notation $X_1\oplus\ldots\oplus
X_r$ is well-defined. We write $X^r=X\oplus\ldots\oplus X$ for the
direct sum of $r$ copies of an object $X$.

Let $X=X_1\oplus\ldots\oplus X_r$ and $Y=Y_1\oplus\ldots\oplus Y_s$ be
two direct sums.  Then one has
\[\Hom_\A(X,Y)=\bigoplus_{i,j}\Hom_\A(X_i,Y_j)\] and therefore each
morphism $\p\colon X\to Y$ can be written uniquely as a matrix
$\p=(\p_{ij})$ with entries $\p_{ij}=\pi_j\p\iota_i$ in
$\Hom_\A(X_i,Y_j)$ for all pairs $i,j$.

A non-zero object $X$ is \emph{indecomposable} if $X=X_1\oplus X_2$
implies $X_1=0$ or $X_2=0$.

An additive category has \emph{split idempotents} if every
idempotent endomorphism $\p=\p^2$ of an object $X$ splits, that is,
there exists a factorisation $X\xto{\pi}Y\xto{\iota}X$ of $\p$ with
$\pi\iota=\id_Y$.

Given an object $X$ in an additive category, we denote by $\add X$ the
full subcategory consisting of all finite direct sums of copies of $X$
and their direct summands. This is the smallest additive subcategory
which contains $X$ and is closed under taking direct summands.

\subsection*{Abelian categories}
An additive category $\A$ is \emph{abelian}, if every morphism
$\p\colon X\to Y$ has a kernel and a cokernel, and if the canonical
factorisation
\[\xymatrix{\Ker\p\ar[r]^-{\p'}&X\ar[r]^-\p\ar[d]&
Y\ar[r]^-{\p''}&\Coker\p\\ 
&\Coker\p'\ar[r]^-{\bar\p}&\Ker\p''\ar[u]}\]
of $\p$ induces  an isomorphism $\bar\p$.

\begin{exm} 
Let $\La$ be an associative ring.

(1) The category $\Mod\La$ of right $\La$-modules is an abelian
category.

(2) The category $\proj\La$ of finitely generated projective
$\La$-modules is an additive category. This category has split
idempotents and equals the subcategory $\add\La$ of $\Mod\La$ which is
given by $\La$ viewed as a $\La$-module.
\end{exm}

\subsection*{Projectivisation}
Every object of an additive category can be turned into a finitely
generated projective module over its endomorphism ring.

\begin{prop}\label{pr:proj}
Let $\A$ be an additive category and $X$ an object with
$\Gamma=\End_\A(X)$. The functor $\Hom_\A(X,-)\colon\A\to\Mod\Gamma$
induces a fully faithful functor $\add X\to\proj\Gamma$. This functor
is an equivalence if $\A$ has split idempotents.
\end{prop}
\begin{proof}
We need to show that $F=\Hom_\A(X,-)$ induces a bijection
\[\Hom_\A(X',X'')\lto\Hom_\Ga(FX',FX'')\] for all $X',X''$ in $\add
X$. Clearly, the map is a bijection for $X'=X=X''$ since
$FX=\Ga$. From this the general case follows because $F$ is additive
and $\proj\Ga=\add\Ga$. Every object in $\proj\Ga$ is a direct summand
of $\Ga^n$ for some $n$ and therefore isomorphic to one in the image
of $F$ if $\A$ has split idempotents.  In that case $F$ induces an
equivalence between $\add X$ and $\proj\Gamma$.
\end{proof}

\begin{rem}\label{re:compl}
Every additive category $\A$ admits an \emph{idempotent completion}
$F\colon\A\to\bar\A$, that is, $\bar\A$ is an additive category with
split idempotents and the functor $F$ is fully faithful, additive, and
each object in $\bar\A$ is a direct summand of an object in the image
of $F$. For instance, if $\A=\add X$ for some object $X$ with
$\Ga=\End_\A(X)$, then one takes $\bar\A=\proj\Ga$ and
$F=\Hom_\A(X,-)$.
\end{rem}

\subsection*{Subobjects}
Let $\A$ be an abelian category. We say that two monomorphisms
$X_1\to X$ and $X_2\to X$ are \emph{equivalent},
if there exists an isomorphism $X_1\xto{\sim} X_2$ 
making the following diagram commutative.
\begin{equation*}
\xymatrix@=1.0em{X_1\ar[rd]\ar[rr]&&X_2\ar[ld]\\&X}
\end{equation*}
An equivalence class of monomorphisms into $X$ is called a
\emph{subobject} of $X$. Given subobjects $X_1\to X$ and $X_2\to X$,
we write $X_1\subseteq X_2$ if there is a morphism $X_1\to X_2$ making
the above diagram commutative. 

An object $X\neq 0$ is \emph{simple} if $X'\subseteq X$ implies $X'=0$
or $X'=X$.

Given a family of subobjects $(X_i)_{i\in I}$ of an object $X$, let
$\sum_{i\in I}X_i$ denote the smallest subobject of $X$ containing all
$X_i$, provided such an object exists. If the coproduct $\coprod_{i\in
I}X_i$ exists in $\A$, then $\sum_{i\in I}X_i$ equals the image of the
canonical morphism $\coprod_{i\in I}X_i\to X$.  The family of
subobjects $(X_i)_{i\in I}$ is \emph{directed} if for each pair
$i,j\in I$, there exists $k\in I$ with $X_{i},X_{j}\subseteq X_k$.

An object $X$ is \emph{finitely generated} if $X=\sum_{i\in I} X_i$
for some directed set of subobjects $X_i\subseteq X$ implies
$X=X_{i_0}$ for some index $i_0\in I$.

\begin{lem}\label{le:fgmax}
Let $X$ be a finitely generated object. Suppose that the subobjects of
$X$ form a set and that $\sum_{i\in I}X_i$ exists for every family of
subobjects $(X_i)_{i\in I}$.  Then every proper subobject of $X$ is
contained in a maximal subobject.
\end{lem}
\begin{proof}
Apply Zorn's lemma.
\end{proof}

\begin{exm}
A $\La$-module $X$ is finitely generated if and only if there exist
elements $x_1,\ldots,x_n$ in $X$ such that $X=\sum_ix_i\La$.
\end{exm}

\subsection*{The Jacobson radical}

Let $X$ be an object in an abelian category. The \emph{radical} of $X$
is the intersection of all its maximal subobjects and is denoted by
$\rad X$. Note that $\p(\rad X)\subseteq\rad Y$ for every morphism
$\p\colon X\to Y$. Thus the assignment $X\mapsto \rad X$ defines a
subfunctor of the identity functor.

For a ring $\La$, the radical of the $\La$-module $\La$ is called
\emph{Jacobson radical} and will be denoted by $J(\La)$. The following
lemma implies that $J(\La)$ is a two-sided ideal.

\begin{lem}[Nakayama]
\label{le:nak}
Let $X$ be a $\La$-module. Then $X J(\La)\subseteq \rad X$.  In
particular, $X J(\La)=X$ implies $X=0$ provided that $X$ is finitely
generated.
\end{lem}
\begin{proof}
For any $x\in X$, left multiplication with $x$ induces a morphism
$\La\to X$, and therefore $x(\rad\La)\subseteq\rad X$.

If $X$ is finitely generated, then every proper submodule is contained
in a maximal submodule. Thus $\rad X=X$ implies $X=0$.
\end{proof}

The next lemma gives a more explicit description of the Jacobson
radical. In particular, one sees that it is a left-right symmetric
concept.

\begin{lem}\label{le:Jrad}
Let $\La$ be a ring. Then
\begin{align*}
J(\La)&=\{x\in\La\mid 1-xy\text{ has a right inverse for all }y\in\La\}\\
&=\{x\in\La\mid 1-y'xy\text{ is invertible for all }y,y'\in\La\}.
\end{align*}
In particular, $J(\La^\op)=J(\La)$.
\end{lem}
\begin{proof}
We have $x\in J(\La)$ if and only if $\frm +x\La\neq\La$ for every
maximal right ideal $\frm$, and this is equivalent to
$1-xy\not\in\frm$ for every $y\in\La$ and maximal $\frm$, that is,
$1-xy$ has a right inverse. This establishes the first equality.

For the second equality, it remains to show that $x\in J(\La)$ implies
$1-x$ is invertible.  We know there exists $z$ such that
$(1-x)z=1$. Thus $1-z=-xz\in J(\La)$, so there exists $z'$ such that
$(1-(1-z))z'=1$, that is, $zz'=1$. Hence $z$ is invertible, and so is
then also $1-x$.
\end{proof}

\subsection*{The radical of an additive category}

Let $\A$ be an additive category.  A \emph{two-sided ideal} $\frI$ of
$\A$ consists of subgroups $\frI(X,Y)\subseteq\Hom_\A(X,Y)$ for each
pair of objects $X,Y\in\A$ such that for every sequence $X'\xto{\s}
X\xto{\p}Y\xto{\t}Y'$ of morphisms in $\A$ with $\p\in\frI(X,Y)$ the
composite $\t\p\s$ belongs to $\frI(X',Y')$. Note that a morphism
$(\p_{ij})\colon\bigoplus_iX_i\to\bigoplus_jY_j$ belongs to an
ideal $\frI$ if and only if $\p_{ij}\in\frI$ for all $i,j$.

Given a pair $X,Y$ of objects of
$\A$, we define the \emph{radical}
\[\Rad_\A(X,Y):=\{\p\in\Hom_\A(X,Y)\mid \p\psi\in J(\End_\A(Y))\text{
  for all }\psi\in\Hom_\A(Y,X)\}.\] It follows from
Lemma~\ref{le:Jrad} that $\p\in\Hom_\A(X,Y)$ belongs to the radical if
and only if $\id_Y-\p\psi$ has a right inverse for every
$\psi\in\Hom_\A(Y,X)$.

\begin{prop} 
\label{le:radcat}
The radical $\Rad_\A$ is the unique two-sided ideal of $\A$
such that $\Rad_\A(X,X)=J(\End_\A(X))$ for every object $X\in\A$.
\end{prop}
\begin{proof}
Each set $\Rad_\A(X,Y)$ is a subgroup of $\Hom_\A(X,Y)$ since
$J(\End_\A(Y))$ is a subgroup of $\End_\A(Y)$.  Now fix a sequence
$X'\xto{\s} X\xto{\p}Y\xto{\t}Y'$ of morphisms in $\A$ with
$\p\in\Rad_\A(X,Y)$. Clearly, $\p\s\in\Rad_\A(X',Y)$ and it remains to
show that $\t\p\in\Rad_\A(X,Y')$. We use the description of the
Jacobson radical in Lemma~\ref{le:Jrad}. Choose
$\psi\in\Hom_\A(Y',X)$. Then $\id_Y-\p\psi\t$ has a right inverse, say
$\a\in\End_\A(Y)$, since $\p\in\Rad_\A(X,Y)$. A simple calculation
shows that $(\id_{Y'}-\t\p\psi)(\id_{Y'}+\t\a\p\psi)=\id_{Y'}$. It
follows that $\t\p$ belongs to $\Rad_\A(X,Y')$. Thus $\Rad_\A$ is a
two-sided ideal of $\A$.

It is clear from the definition that $\Rad_\A(X,X)=J(\End_\A(X))$ for
every $X\in\A$. Any two-sided ideal $\frI$ of $\A$ is determined by
the collection of subgroups $\frI(X,X)$, where $X$ runs through all
objects of $\A$. In fact, a morphism $\p\in\Hom_\A(X,Y)$ belongs to
$\frI(X,Y)$ if and only if $\smatrix{0&0\\ \p&0}$ belongs to $\frI(X\oplus
Y,X\oplus Y)$.
\end{proof}

The following description of the radical $\Rad_\A$ is a consequence;
it is symmetric and shows that $\Rad_{\A^\op}=\Rad_\A$.

\begin{cor}
Let $X,Y$ be a pair of objects of an additive category $\A$. Then the
following are equivalent for a morphism $\p\colon X\to Y$.
\begin{enumerate}
\item $\p\in\Rad_\A(X,Y)$.
\item $\id_Y-\p\psi$ has a right inverse for all morphisms
$Y\xto{\psi}X$.
\item $\t\p\s\in J(\End_\A(Z))$  for all morphisms
$Y\xto{\t}Z\xto{\s}X$.
\item $\id_Z-\t\p\s$ is invertible for all morphisms
$Y\xto{\t}Z\xto{\s}X$.\qed
\end{enumerate}
\end{cor}

\section{Projective covers}

\subsection*{Essential epimorphisms}

Let $\A$ be an abelian category.  An epimorphism $\p\colon X\to Y$ is
\emph{essential} if any morphism $\a\colon X'\to X$ is an epimorphism
provided that the composite $\p\a$ is an epimorphism. This condition
can be rephrased as follows: If $U\subseteq X$ is a subobject with
$U+\Ker\p=X$, then $U=X$. We collect some basic facts.

\begin{lem}\label{le:compess}
Let $\p\colon X\to Y$ and $\psi\colon Y\to Z$ be epimorphisms. Then
$\psi\p$ is essential if and only if both $\p$ and $\psi$ are essential.
\qed\end{lem}

\begin{lem}\label{le:sumess}
Let $X_i\to Y_i$ $(i=1,\ldots,n)$ be essential epimorphisms. Then
$\bigoplus_{i}X_i\to\bigoplus_{i}Y_i$ is an essential epimorphism.
\end{lem}
\begin{proof}
It is sufficient to pove the case $n=2$.  In that case write
$\bigoplus_{i}X_i\to\bigoplus_{i}Y_i$ as composite $X_1\oplus X_2\to
X_1\oplus Y_2\to Y_1\oplus Y_2$. It is straightforward to check that
both morphisms are essential. Thus the composite is essential.
\end{proof}

\begin{lem}\label{le:radess}
Let $\p\colon X\to Y$ be an epimorphism and  $U=\Ker\p$.
\begin{enumerate}
\item If $\p$ is essential, then $U \subseteq \rad X$.
\item If $U\subseteq \rad X$ and  $X$ is finitely generated, then $\p$ is essential.  
\end{enumerate}
\end{lem}
\begin{proof}
(1) Suppose that $\p$ is essential and let $V\subseteq X$ be a maximal
subobject not containing $U$. Then $U+V=X$ and therefore $V=X$. This
is a contradiction and therefore $U$ is contained in every maximal
subobject. This implies $U\subseteq\rad X$.

(2) Suppose that $U\subseteq\rad X$ and let $V\subseteq X$ be a
subobject with $U+V=X$.  If $V\neq X$, then there is a maximal
subobject $V'\subseteq X$ containing $V$ since $X$ is finitely
generated; see Lemma~\ref{le:fgmax}. Thus $X=U+V\subseteq V'$. This is a contradiction and
therefore $V=X$.  It follows that $\p$ is essential.
\end{proof}

\subsection*{Projective covers}

Let $\A$ be an abelian category.  An epimorphism $\p\colon P\to X$ is
called a \emph{projective cover} of $X$ if $P$ is projective and $\p$
is essential.

\begin{lem}\label{le:procov}
Let $P$ be a projective object. Then the following are equivalent for
an epimorphism $\p\colon P\to X$.
\begin{enumerate}
\item The morphism $\p$ is a projective cover of $X$.
\item Every endomorphism $\a\colon P\to P$ satisfying $\p\a=\p$
is an isomorphism.
\end{enumerate}
\end{lem}
\begin{proof}
(1) $\Rightarrow$ (2): Let $\a\colon P\to P$ be an endomorphism
    satisfying $\p\a=\p$. Then $\a$ is an epimorphism since $\p$
    is essential. Thus there exists $\a'\colon P\to P$ satisfying
    $\a\a'=\id_P$ since $P$ is projective.  It follows that
    $\p\a'=\p$ and therefore $\a'$ is an epimorphism. On the
    other hand, $\a'$ is a monomorphism. Thus $\a'$ and $\a$ are
    isomorphisms.

(2) $\Rightarrow$ (1): Let $\a\colon P'\to P$ be a morphism such that
    $\p\a$ is an epimorphism. Then $\p$ factors through
    $\p\a$ via a morphism $\a'\colon P\to P'$ since $P$ is
    projective. The composite $\a\a'$ is an isomorphism and
    therefore $\a$ is an epimorphism. Thus $\p$ is essential.
\end{proof}

\begin{cor}\label{co:cov}
Let $\p\colon P\to X$ and $\p'\colon P'\to X$ be projective covers of
an object $X$. Then there is an isomorphism $\a\colon P\to P'$ such
that $\p=\p'\a$.
\end{cor}

A ring is called \emph{local} if the sum of two non-units is again a
non-unit.

\begin{lem}
\label{le:covsimple}
Let $\p\colon P\to S$ be an epimorphism such that $P$ is projective
and $S$ is simple.  Then the following are equivalent.
\begin{enumerate}
\item The morphism $\p$ is a projective cover of $S$.
\item The object $P$ has a maximal subobject that contains every
  proper subobject of $P$.
\item The endomorphism ring of $P$ is local.
\end{enumerate}
\end{lem}
\begin{proof}
(1) $\Rightarrow$ (2): Let $U\subseteq P$ be a subobject and suppose
    $U\not\subseteq \Ker\p$. Then $U+\Ker\p=P$, and therefore $U=P$
    since $\p$ is essential.  Thus $\Ker\p$ contains every proper
    subobject of $P$.

(2) $\Rightarrow$ (3): First observe that $P$ is an indecomposable
object.  It follows that every endomorphism of $P$ is invertible if
and only if it is an epimorphism. Given two non-units $\a,\b$ in
$\End_\A(P)$, we have therefore
$\Im(\a+\b)\subseteq\Im\a+\Im\b\subseteq\rad P$.  Here we use that
$\rad P$ contains every proper subobject of $P$. Thus $\a+\b$ is a
non-unit and $\End_\A(P)$ is local.

(3) $\Rightarrow$ (1): Consider the $\End_\A(P)$-submodule $H$ of
$\Hom_\A(P,S)$ which is generated by $\p$. Suppose $\p=\p\a$ for some
$\a$ in $\End_\A(P)$.  If $\a$ belongs to the Jacobson radical, then
$H=H J(\End_\A(P))$, which is not possible by Lemma~\ref{le:nak}. Thus
$\a$ is an isomorphism since $\End_\A(P)$ is local. It follows from
Lemma~\ref{le:procov} that $\p$ is a projective cover.
\end{proof}

\subsection*{Maximal subobjects of projectives}

Let $\A$ be an abelian category.  We need to assume that for each
object $X$ the subobjects of $X$ form a set and that $\sum_{i\in
I}X_i$ exists for each family of subobjects $(X_i)_{i\in I}$.  Given a
subobject $U\subseteq X$, we set
\[\End_\A(U|X):=\{\p\in\End_\A(X)\mid\Im\p\subseteq U\}.\] 

\begin{prop}\label{pr:maxsub}
Let $\A$ be an abelian category and $X$ a finitely generated
projective object. The maps
\[X\supseteq
U\mapsto\End_\A(U|X)\quad\text{and}\quad
\End_\A(X)\supseteq\fra\mapsto\sum_{\End_\A(V|X)\subseteq\fra}V\]
induces mutually inverse bijections between the maximal subobjects of
$X$ and the maximal right ideals of $\End_\A(X)$.
\end{prop}
\begin{proof}
A subobject $U\subseteq X$ induces an exact sequence
\[0\to\Hom_\A(X,U)\to\Hom_\A(X,X)\to\Hom_\A(X,X/U)\to 0.\] If
$U\subseteq X$ is maximal, then $\Hom_\A(X,X/U)$ is a simple
$\End_\A(X)$-module and therefore $\End_\A(U|X)$ is a maximal right ideal.

Now fix a maximal right ideal $\frm$ of $\End_\A(X)$ and let $U=\sum_{V\in\V}
V$ where $\V$ denotes the set of subobjects
$V\subseteq X$ with $\End_\A(V|X)\subseteq\frm$.  First notice that
$\frm\subseteq\End_\A(U|X)$ since $X$ is projective.  Next observe
that $\V$ is directed since $V_1,V_2\in\V$ implies $V_1+V_2\in\V$.
Thus $U$ is a proper subobject of $X$ since $X$ is finitely
generated. In particular, $\End_\A(U|X)=\frm$.  If $W\subseteq X$ is a
subobject properly containing $U$, then $\End_\A(W|X)$ properly
contains $\frm$ and equals therefore $\End_\A(X)$. Thus $W=X$. It
follows that $U$ is maximal.
\end{proof}

\begin{cor}\label{co:projrad}
Let $\p\colon X\to Y$ be a morphism and suppose $Y$ is finitely
generated projective. Then $\Im\p\subseteq\rad Y$ if and only if
$\p\in\Rad_\A(X,Y)$.
\end{cor}
\begin{proof}
  We apply Proposition~\ref{pr:maxsub}.  For every maximal subobject
  $U\subseteq Y$, we have $\Im\p\subseteq U$ if and only
  if \[\Im\Hom_\A(Y,\p)=\End_\A(\Im\p|X)\subseteq
\End_\A(U|Y).\] Thus
$\Im\p\subseteq\rad Y$ if and only if $\Im\Hom_\A(Y,\p)\subseteq
J(\End_\A(Y))$. It follows that $\p$ belongs to $\Rad_\A(X,Y)$.
\end{proof}

\begin{rem}
The assumption on $Y$ to be projective is necessary in
Corollary~\ref{co:projrad}.  Take for instance over
$\La=k[x,y]/(x^2,y^2)$ ($k$ any field) the module $Y=\rad\La$ and let
$\p\colon X\to Y$ be the inclusion of a maximal submodule $X$. Then
$\p\in\Rad_\La(X,Y)$ but $\Im\p=X\not\subseteq\rad Y$.
\end{rem}

\subsection*{Projective presentations}

Let $\A$ be an abelian category. An exact sequence $P_1\xto{}
P_0\xto{} X\to 0$ is called a \emph{projective presentation} of $X$ if
$P_0$ and $P_1$ are projective objects.

\begin{prop}
Let $P_1\xto{\p} P_0\xto{\psi} X\to 0$ be a projective
presentation. Then $\psi$ is a projective cover of $X$ if and only if
$\p$ belongs to $\Rad_\A(P_1,P_0)$.
\end{prop}
\begin{proof}
Let $P=P_0\oplus P_1$ and $\Ga=\End_\A(P)$. Denote by $\C$ the
smallest full additive subcategory of $\A$ containing $P$ and closed
under taking cokernels. Using Proposition~\ref{pr:proj}, it is not
hard to verify that $F=\Hom_\A(P,-)\colon\A\to\Mod\Ga$ induces an
equivalence $\C\xto{\sim}\mod\Ga$, where $\mod\Ga$ denotes the
category of finitely presented $\Ga$-modules.

It follows from Lemma~\ref{le:procov} that $\psi$ is a projective
cover of $X$ if and only if $F\psi$ is a projective cover of $FX$.
The module $FP_0$ is finitely generated and therefore $F\psi$ is a
projective cover if and only if $\Ker F\psi\subseteq\rad FP_0$, by
Lemma~\ref{le:radess}. Finally, Corollary~\ref{co:projrad} implies
that $\Ker F\psi\subseteq\rad FP_0$ if and only if $F\p$ belongs to
$\Rad_\Ga(FP_0,FP_1)$. It remains to note that $F$ induces a
bijection $\Rad_\A(P_0,P_1)\cong\Rad_\Ga(FP_0,FP_1)$.
\end{proof}

\section{Krull-Schmidt categories}

\subsection*{Krull-Schmidt categories}

An additive category is called \emph{Krull-Schmidt category} if every
object decomposes into a finite direct sum of objects having local
endomorphism rings.

\begin{prop}\label{pr:semiperfect}
For a ring $\La$ the following are equivalent.
\begin{enumerate}
\item The category of finitely generated projective
$\La$-modules is a Krull-Schmidt category.
\item The module $\La$ admits a decomposition
$\La=P_1\oplus\ldots\oplus P_r$ such that each $P_i$ has a local
endomorphism ring.
\item Every simple $\La$-module admits a projective cover.
\item Every finitely generated $\La$-module admits a projective cover.
\end{enumerate}
\end{prop}

A ring is \emph{semi-perfect} if it satisfies the equivalent
conditions in the preceding proposition.

\begin{proof}
(1) $\Rightarrow$ (2): Clear.

(2) $\Rightarrow$ (3): Let $S$ be a simple $\La$-module. Then we have
a non-zero morphism $\La\to S$ and therefore a non-zero morphism
$\p\colon P\to S$ for some indecomposable direct summand $P$ of
$\La$. The morphism $\p$ is a projective cover by
Lemma~\ref{le:covsimple}, because $\End_\La(P)$ is local.

(3) $\Rightarrow$ (1): Let $P$ be a finitely generated projective
$\La$-module. We claim that $P/\rad P$ is semi-simple. To prove this,
let $P'/\rad P\subseteq P/\rad P$ be the sum of all simple
submodules. If $P'\neq P$, there is a maximal submodule $U\subseteq P$
containing $P'$, and the simple module $P/U$ admits a projective cover
$\pi\colon Q\to P/U$. The morphism $P\to P/U$ factors through $Q\to
P/U$ via a morphism $\p\colon P\to Q$. Analogously, there is a
morphism $\psi\colon Q\to P$, and the composite $\p\psi$ is an
isomorphism since $\pi$ is a projective cover, by
Lemma~\ref{le:procov}.  Observe that $\Ker \pi=\rad Q$, by
Lemma~\ref{le:covsimple}. Thus $P/U\cong Q/\rad Q$, and therefore
$\psi$ induces a right inverse for the canonical morphism $P/\rad P\to
P/U$. This contradicts the property of $P'/\rad P$ to contain all
simple submodules of $P/\rad P$. It follows that $P/\rad P$ is
semi-simple. Let $P/\rad P=\bigoplus_i S_i$ be a decomposition into
finitely many simple modules and choose a projective cover $P_i\to
S_i$ for each $i$. Then $P\cong \bigoplus_i P_i$, since $P\to P/\rad
P$ and $ \bigoplus_i P_i\to \bigoplus_i S_i$ are both projective
covers. It remains to observe that each $P_i$ is indecomposable with a
local endomorphism ring, by Lemma~\ref{le:covsimple}. 

(1) \& (3) $\Rightarrow$ (4): The assumption implies that every finite
sum of simple $\La$-modules admits a projective cover; see
Lemma~\ref{le:sumess}.  Now let $X$ be a finitely generated $\La$-module
and choose an epimorphism $\p\colon P\to X$ with $P$ finitely
generated projective. Let $P=\bigoplus_{i=1}^nP_i$ be a decomposition
into indecomposable modules. Then
\[P/\rad P=\bigoplus_{i=1}^nP_i/\rad P_i\] is a finite sum of simple
$\La$-modules by Lemma~\ref{le:covsimple} since each $P_i$ has a local
endomorphism ring. The epimorphism $\p$ induces an epimorphism $P/\rad
P\to X/\rad X$ and therefore $X/\rad X$ decomposes into finitely many
simple modules.  There exists a projective cover $Q\to X/\rad X$ and
this factors through the canonical morphism $\pi\colon X\to X/\rad X$
via a morphism $\psi \colon Q\to X$. The morphism $\psi$ is an
epimorphism because $\pi$ is essential by Lemma~\ref{le:radess}, and
Lemma~\ref{le:compess} implies that $\psi$ is essential.

(4) $\Rightarrow$ (3): Clear.
\end{proof}

\subsection*{Direct sum decompositions}

The uniqueness of direct sum decompositions in Krull-Schmidt
categories can be derived from the existence and uniqueness of
projective covers over semi-perfect rings.

\begin{thm}\label{th:unique}
Let $X$ be an object of an additive category and suppose there are two
decompositions \[X_1\oplus\ldots\oplus X_r=X= Y_1\oplus\ldots \oplus
Y_s\] into objects with local endomorphism rings.  Then $r=s$ and and
there exists a permutation $\pi$ such that $X_i\cong Y_{\pi(i)}$ for
$1\le i\le r$.
\end{thm}
\begin{proof}
Let $\A=\add X$ and identify $\A$ via $\Hom_\A(X,-)$ with a full
subcategory of the category of finitely generated projective modules
over $\End_\A(X)$; see Proposition~\ref{pr:proj}. Thus we may assume
that $X$ is a finitely generated projective module over a semi-perfect
ring.

It follows from Lemma~\ref{le:covsimple} that for every index $i$ the
radical $\rad X_i$ is a maximal submodule of $X_i$ and that the
canonical morphism $X_i\to X_i/\rad X_i$ is a projective cover. Thus
$X_i\cong Y_j$ if and only if $X_i/\rad X_i\cong Y_j/\rad Y_j$ for
every pair $i,j$, by Corollary~\ref{co:cov}.  We have
\[(X_1/\rad X_1)\oplus\ldots\oplus (X_r/\rad X_r)=X/\rad X= (Y_1/\rad
Y_1)\oplus\ldots\oplus (Y_s/\rad Y_s)\] and the assertion now follows
from the uniqueness of the decomposition of a semi-simple module into
simple modules (which is easily proved by induction on the number of
summands).
\end{proof}

\begin{cor}\label{co:exchange}
  Let $X$ be an object of a Krull-Schmidt category and suppose there
  are two decompositions \[X_1\oplus\ldots\oplus X_n=X= X' \oplus
  X''\] such that each $X_i$ is indecomposable.  Then there exists an
  integer $t\le n$ such that $X=X_1\oplus\ldots\oplus X_t\oplus X'$
  after reindexing the $X_i$.
\end{cor}
\begin{proof}
  Let $X'=Y_1\oplus\ldots \oplus Y_s$ and $X''=Z_1\oplus\ldots \oplus
  Z_t$ be decompositions into indecomposable objects.  It follows from
  the uniqueness of these decompositions that $n=s+t$ and that
  $X''\cong X_1\oplus\ldots\oplus X_t$ after some reindexing of the
  $X_i$. Composing the decomposition $X= X' \oplus X''$ with that
  isomorphism yields the assertion.
\end{proof}

\begin{cor}\label{co:semi-perfect}
  An additive category is a Krull-Schmidt category if and only if it
  has split idempotents and the endomorphism ring of every object is
  semi-perfect.
\end{cor}
\begin{proof} 
  The assertion follows from Proposition~\ref{pr:semiperfect} once we
  know that a Krull-Schmidt category has split idempotents. But this
  is clear since there is an equivalence $\add X\xto{\sim}\proj\Ga$
  for $\Ga=\End_\A(X)$, thanks to Proposition~\ref{pr:proj} and
  Corollary~\ref{co:exchange}.
\end{proof}

Let $\A$ be a Krull-Schmidt category and let
$X=X_1\oplus\ldots\oplus X_r$ and $Y=Y_1\oplus\ldots\oplus Y_s$ be
decompositions of two objects $X,Y$ into indecomposable objects.  Then
we have
\[\Rad_\A(X,Y)=\bigoplus_{i,j}\Rad_\A(X_i,Y_j)\] and
$\Rad_\A(X_i,Y_j)$ equals the set of non-invertible morphisms $X_i\to
Y_j$ for each pair $i,j$.

\begin{exm}
  The category of finitely generated torsion-free abelian groups
  admits unique decompositions into indecomposable objects. However,
  the unique indecomposable object $\bbZ$ does not have a local
  endomorphism ring.
\end{exm}

\section{Chain conditions}

\subsection*{The bi-chain condition}
A \emph{bi-chain} in a category is a sequence of morphisms
$X_n\xto{\a_n} X_{n+1} \xto{\b_n} X_{n}$ ($n\ge 0$) such that $\a_n$
is an epimorphism and $\b_n$ is a monomorphism for all integers $n\ge
0$. The object $X$ satisfies the \emph{bi-chain condition} if for
every bi-chain $X_n\xto{\a_n} X_{n+1} \xto{\b_n} X_{n}$ ($n\ge 0$)
with $X=X_0$ there exists an integer $n_0$ such that $a_n$ and $\b_n$
are invertible for all $n\ge n_0$.

\subsection*{Finite length objects}
An object $X$ of an abelian category has \emph{finite length} if there
exists a finite chain of subobjects
\[0=X_0\subseteq X_1\subseteq \ldots \subseteq X_{n-1}\subseteq
X_n=X\] such that each quotient $X_i/X_{i-1}$ is a simple object. Note
that $X$ has finite length if and only if $X$ is both \emph{artinian}
(i.e.\ it satisfies the descending chain condition on subobjects) and
\emph{noetherian} (i.e.\ it satisfies the ascending chain condition on
subobjects).

\begin{lem}
An object of finite length satisfies the bi-chain condition.
\end{lem}
\begin{proof}
  Let $X$ be an object of finite length and $X_n\xto{\a_n} X_{n+1}
  \xto{\b_n} X_{n}$ ($n\ge 0)$ a bi-chain with $X=X_0$. Then the
  subobjects $\Ker (\a_n\ldots\a_1\a_0)\subseteq X$ yield an ascending
  chain and the subobjects $\Im (\b_0\b_1\ldots\b_n)\subseteq X$ yield
  a descending chain. If these chains terminate, then $\a_n$ and
  $\b_n$ are invertible for large enough $n$.
\end{proof}

An additive category $\A$ is \emph{Hom-finite} if there exists a
commutative ring $k$ such that $\Hom_\A(X,Y)$ is a $k$-module of
finite length for all objects $X,Y$ and the composition maps are
$k$-bilinear.

\begin{lem}
  An object of a Hom-finite abelian category satisfies the bi-chain
  condition.
\end{lem}
\begin{proof}
  Let $X$ be an object of a Hom-finite abelian category $\A$ and
  $X_n\xto{\a_n} X_{n+1} \xto{\b_n} X_{n}$ ($n\ge 0)$ a bi-chain with
  $X=X_0$. Each pair $\a_n,\b_n$ induces a monomorphism
  $\Hom_\A(X_{n+1},X_{n+1})\to \Hom_\A(X_{n},X_{n})$. If this map is
  bijective, then $\a_n$ is a monomorphism and $\b_n$ is an
  epimorphism. In an abelian category, any morphism is invertible if
  it is both a monomorphism and an epimorphism. Thus the assumption on
  $\A$ implies that $\a_n$ and $\b_n$ are invertible for large enough
  $n$.
\end{proof}

\subsection*{Fitting's lemma}
Fix an  abelian category. 
\begin{lem}[Fitting] 

\label{le:fitting}
Let $X$ be an object satisfying the bi-chain condition and $\p$ an
endomorphism.
\begin{enumerate}
\item For large enough $r$, one has  $X=\Im\p^r\oplus\Ker\p^r$.
\item If $X$ is indecomposable, then $\p$ is either invertible or
nilpotent.
\end{enumerate}
\end{lem}
\begin{proof}
  The endomorphism $\p$ yields a bi-chain $X_n\xto{\a_n} X_{n+1}
  \xto{\b_n} X_{n}$ ($n\ge 0)$ with $X_n=\Im\p^n$, $\a_n=\p$, and
  $\b_n$ the inclusion.  Because $X$ satisfies the bi-chain condition,
  we may choose $r$ large enough so that $\Im \p^r=\Im\p^{r+1}$. Thus
  $\p^r\colon\Im\p^r\to\Im\p^{2r}$ is an isomorphism and we denote by
  $\psi$ its inverse.  Furthermore, let $\iota_1\colon\Im\p^r\to X$
  and $\iota_2\colon \Ker\p^r\to X$ denote the inclusions. We put
  $\pi_1=\psi\p^r\colon X\to \Im\p^r$ and $\pi_2=\id_X-\psi\p^r\colon
  X\to \Ker\p^r$. Then $\iota_1\pi_1+\iota_2\pi_2=\id_X$ and
  $\pi_i\iota_i=\id_{X_i}$ for $i=1,2$. Thus
  $X=\Im\p^r\oplus\Ker\p^r$.  Part (2) is an immediate consequence of
  (1).
\end{proof}

\begin{prop}
\label{pr:fitting}
An object satisfying the bi-chain condition is indecomposable if and
only if its endomorphism ring is local.
\end{prop}
\begin{proof}
Let $X$ be an indecomposable object and $\p,\p'$ a pair of
endomorphisms.  Suppose $\p+\p'$ is invertible, say
$\r(\p+\p')=\id_X$. If $\p$ is non-invertible then $\r\p$ is
non-invertible. Thus $\r\p$ is nilpotent, say $(\r\p)^r=0$, by
Lemma~\ref{le:fitting}. We obtain
\[(\id_X-\r\p)(\id_X+\r\p+\ldots+(\r\p)^{r-1})=\id_X.\] Therefore
$\r\p'=\id_X-\r\p$ is invertible whence $\p'$ is invertible.

If $X=X_1\oplus X_2$ with $X_i\neq 0$ for $i=1,2$, then we have
idempotent endomorphisms $\e_i$ of $X$ with $\Im\e_i=X_i$. Clearly,
each $\e_i$ is non-invertible but $\id_X=\e_1+\e_2$.
\end{proof}

\subsection*{Krull-Remak-Schmidt decompositions}

Fix an abelian category.

\begin{thm}[Atiyah]\label{th:bichain}
An object  satisfying the bi-chain condition admits a decomposition
into a finite direct sum of indecomposable objects having local
endomorphism rings.
\end{thm}
\begin{proof}
  Fix an object $X$ satisfying the bi-chain condition. Assume that $X$
  has no decomposition into a finite direct sum of indecomposable
  objects. Then there is a decomposition $X=X_1\oplus Y_1$ such that
  $X_1$ has no decomposition into a finite direct sum of
  indecomposable objects and $Y_1\neq 0$. We continue decomposing
  $X_1$ and obtain a bi-chain $X_n\xto{\a_n} X_{n+1} \xto{\b_n} X_{n}$
  ($n\ge 0$) with $X=X_0$ and $\a_n\b_n=\id_{X_{n+1}}$ for all $n\ge
  0$. This bi-chain does not terminate and this is a contradiction.

  It remains to observe that any direct summand of $X$ satisfies the
  bi-chain condition. In particular, every indecomposable direct
  summand has a local endomorphism ring by
  Proposition~\ref{pr:fitting}.
\end{proof}

\subsection*{Acknowldgements}
I wish to thank Andrew Hubery for various helpful comments.  Also, I
am grateful to Geordie Williamson for encouraging me to publish these
notes.

\end{document}